\documentclass[10pt,a4paper,reqno]{amsart}

\usepackage{graphicx}     
\usepackage{rawfonts}
\usepackage{multicol}        
\usepackage{amsmath,amssymb,amsfonts}        

\usepackage[all]{xy}

\usepackage{hyperref}
\usepackage{todonotes}

\usepackage[mathscr]{euscript}

\usepackage{fullpage}

\usepackage{braket}

\usepackage{MnSymbol}

\usepackage{pgf,tikz}
\usepackage{mathrsfs}
\usetikzlibrary{arrows}
\definecolor{cqcqcq}{rgb}{0.7529411764705882,0.7529411764705882,0.7529411764705882}
\definecolor{qqzzff}{rgb}{0.,0.6,1.}
\definecolor{wwccqq}{rgb}{0.4,0.8,0.}
\definecolor{ffqqqq}{rgb}{1.,0.,0.}
\input{prepictex.tex}
\input{pictex.tex}
\input{postpictex.tex}
\usepackage{float}

\newcommand{\A}{\mathcal{A}}
\newcommand{\B}{\mathcal{B}}

\newcommand{\Ms}{\mathcal{M}}

\newcommand{\cK}{\mathcal{K}}
\newcommand{\CP}{\mathbb{C}P}

\newcommand{\zz}{\mathbb{Z}}
\newcommand{\cc}{\mathbb{C}}

\newcommand{\nn}{\mathbb{N}}

\newcommand{\Cs}{C$^*$-}

\DeclareMathOperator{\id}{id}

\newtheorem{thm}{Theorem}[section]
\newtheorem{lem}[thm]{Lemma}
\newtheorem{prop}[thm]{Proposition}

\theoremstyle{definition}
\newtheorem{defn}{Definition}[section]

\newtheorem{rem}{Remark}[section]

\title{Split extensions and $KK$-equivalences for quantum projective spaces}
\author{Francesca Arici}
\address{Mathematical Institute, Leiden University, \\
P.O. Box 9512, 2300 RA
Leiden, \\
The Netherlands.}
\email{f.arici@math.leidenuniv.nl}

\author{Sophie Emma Zegers}
\address{Faculty of Mathematics and Physics, Charles University, Sokolovská 49/83, 186 75 Praha 8, Czech Republic}
\email{zegers@karlin.mff.cuni.cz, sophieemmazegers@gmail.com} 

\date{\today}
\makeatletter
\@namedef{subjclassname@2020}{%
  \textup{2020} Mathematics Subject Classification}
\makeatother
\subjclass[2020]{19K35, 46L85, 46L65, 58B34}

\keywords{Quantum projective space, $KK$-theory, $KK$-equivalence, Graph algebras}

\begin{document}

\maketitle

\begin{abstract}
 We study the noncommutative topology of the $C^*$-algebras $C(\CP^{n}_q)$ of the quantum projective spaces within the framework of Kasparov’s bivariant $K$-theory. In particular, we construct an explicit $KK$-equivalence with the commutative algebra $\mathbb{C}^{n+1}$.  
  Our construction relies on showing that the extension of $C^*$-algebras relating two quantum projective spaces of successive dimensions admits a splitting, which we can describe explicitly using graph algebra techniques.
\end{abstract}

\section{Introduction}

Gelfand duality,  which lies at the base of noncommutative geometry, establishes an equivalence of categories between commutative $C^*$-algebras and locally compact Hausdorff spaces.  For this reason, when studying general noncommutative $C^*$-algebras,  even though there is no longer an underlying space, one often thinks of them as algebras of continuous functions on a non-existing virtual space.  

This approach is particularly effective when working with so-called \emph{quantum deformations} of spaces: many classical topological spaces have a $q$-deformed analogs, obtained from quantum groups and their homogeneous spaces. 

The $C^*$-algebra of the quantum $(2n+1)$-sphere by Vaksman and Soibelman \cite{VS91}, denoted $C(S_q^{2n+1})$, is perhaps one of the most studied noncommutative spaces within this class. It is constructed as a quantum homogeneous space for the special unitary group, and can also be proven to be isomorphic to a universal $C^*$-algebra in $(n+1)$ generators subject to a set of commutation relations.  In those relations, a parameter $q\in (0,1)$ plays a central role, making the resulting $C^*$-algebra noncommutative. When writing $C(S_q^{2n+1})$ for the \Cs algebra of the quantum sphere, one often thinks of $S_q^{2n+1}$ as a virtual space. In the limit $q=1$, the resulting \Cs algebra $C(S_1^{2n+1})$ is commutative and isomorphic to the $C^*$-algebra $C(S^{2n+1})$ of continuous functions on the $(2n+1)$-sphere.  

Like their classical counterparts, the odd quantum spheres are endowed with a canonical $U(1)$-action,  allowing one to define,  in total analogy with the commutative setting, the quantum complex projective space $C(\CP_q^n)$ as the fixed point algebra for that action. 
In \cite{HS02}, Hong and Szyma\'nski showed that both $C(S^{2n+1}_q)$ and $C(\CP_q^{n})$ are graph $C^*$-algebras. Through the graph-algebraic picture, one can obtain useful information about the structure of those $C^*$-algebra, including topological invariants, by only considering properties of the underlying graph. In particular, using graph $C^*$-algebra techniques, the $K$-theory groups of  $C(\CP_q^{n})$ can be found to agree with the ones of their classical counterparts:
\[K_0(C(\CP_q^n))\cong \mathbb{Z}^{n+1}, \qquad  K_1(C(\CP_q^n))\cong \{0\}.\]

In \cite{HS02}, Hong and Szyma\'nski further show that for $n\geq 1$, the algebras of two projective spaces of successive dimension fit into an extension of \Cs algebras of the from
\begin{equation}\label{exactseqn}
\xymatrix{0 \ar[r] & \cK \ar[r] & C(\CP^n_q) \ar[r] & C(\CP^{n-1}_q) \ar[r] & 0},
\end{equation}
with the convention that $C(\CP^{n-1}_q) \simeq \mathbb{C}$. It is worth stressing that the exact sequence for $n=1$,
\begin{equation}\label{exactseq1}
\xymatrix{0 \ar[r] & \cK \ar[r] & C(\CP^1_q) \ar[r] & \mathbb{C} \ar[r] & 0},
\end{equation} 
is known to split, which implies that the algebra $C(\CP^1_q)$ is isomorphic to the minimal unitization of the compacts. Note also that one can also compute the $K$-theory groups of quantum projective spaces inductively, using the above exact sequence \eqref{exactseqn}.

In the present work, we bring the analysis of the topological invariants of quantum projective spaces further and study these algebras within the framework of Kasparov's bivariant K-theory \cite{Kas80}.  In particular, we construct an \emph{explicit} $KK$-equivalence between the algebras $C(\CP_q^{n})$ and $\mathbb{C}^{n+1}$. 

The fact that the two algebras are $KK$-equivalent follows from the work \cite{NesTus12}, where Neshveyev and Tuset study quantum homogeneous space $G_q/K_q$. Those are $q$-deformations of the homogenous space $G/K$ for $G$ a compact simply connect semisimple Lie group with arbitrary closed Poisson-Lie subgroup, with deformation parameter $q \in (0,1]$. In particular, Neshveyev and Tuset prove that $q \mapsto C(G_q/K_q)$ is a continuous field of \Cs-algebras, and that all the \Cs algebras in the fiber are canonically $KK$-equivalent, and hence $KK$-equivalent to the commutative algebra $C(G_1/K_1)\cong C(G/K)$. The desired $KK$-equivalence for $C(\CP^n_q)$ follows by considering $G_q=SU_q(n+1)$ and $K_q=U_q(n)$.

Another way to deduce $KK$-equivalence is by looking at the $K$-theory groups of the \Cs algebra $C(\CP^n_q)$ and of its commutative counterparts. Since those are isomorphic, such a $KK$-equivalence follows provided $C(\CP_q^{n})$ is contained in the class of \Cs algebras that satisfy the Universal Coefficient Theorem (UCT) of Rosenberg and Schochet. Indeed, by \cite[Corollary~7.5]{RoSch87} (compare, \cite[Corollary 23.10.2]{Bl98}), two $C^*$-algebras in the UCT class are $KK$-equivalent if and only if they have isomorphic K-theory groups. 

To see that the algebra $C(\CP_q^n)$ is in the UCT class, one observes that the UCT class is closed under extensions and contains the algebra of compact operators and the complex numbers. Hence, by the family of extensions \eqref{exactseqn} and induction on $n$, it follows that $C(\CP_q^n)$ is in the UCT class for all $n$. 

Our strategy for obtaining an explicit $KK$-equivalence between $C(\CP_q^{n})$ and $\mathbb{C}^{n+1}$ consists of proving that \eqref{exactseqn} splits.  For any split exact sequence one obtains a $KK$-class implementing the desired $KK$-equivalence through the so-called splitting homomorphism (see, for instance, \cite[Exercise 19.9.1]{Bl98}).  

The existence of a splitting is, once more, a direct consequence of the Universal Coefficient Theorem \cite{RoSch87}. Knowing that a splitting exists is however not enough for practical applications, and, in general, constructing such a splitting explicitly is a non-trivial task. When considering $C(\CP_q^n)$ as a graph $C^*$-algebra, the structure of the graph makes it easier to unravel the form of such a splitting. To our knowledge, such a splitting has not been described in the literature before.

The structure of the paper is as follows: In Section~\ref{sec:graph} we recall definitions and results on graph \Cs algebras, focusing on K-theory and ideal structure, which we then specialize to quantum projective spaces in  Section~\ref{sec:QPS}. Section~\ref{SplittingProjectivespace} contains the construction of a splitting for the extension \eqref{exactseqn}. We then recall how one obtains explicit $KK$-equivalences from split extensions in Section~\ref{sec:KKequiv}, and then proceed to the proof of our $KK$-equivalence result in Section~\ref{explicitKKequiv}. Finally, in  Section~\ref{ProjectionsFredholmModules} we relate the classes in $KK(\cc, C(\CP_q^n))$ obtained from the splitting to classes of projections in $C(\CP_q^n)$ which generate $K_0(C(\CP_q^n))$. 
\subsection*{Acknowledgements}
We are indebted to our colleagues Bram Mesland and Wojciech Szyma\'nski for inspiring conversations on $KK$-theory and on graph $C^*$-algebras. We would also like to thank Francesco D'Andrea,  Piotr Hajac, Tomasz Maszczyk, and Sergey Neshveyev for valuable comments on an earlier version of this work. We also thank the anonymous referee for comments and suggestions for improvement. This work is part of the research program VENI with project number 016.192.237, which is (partly) financed by the Dutch Research Council (NWO). The second author was supported by the DFF-Research Project 2 on `Automorphisms and invariants of operator algebras', Nr. 7014--00145B and by the Carlsberg Foundation through an Internationalisation Fellowship.
  
  \section{Preliminaries on graph algebras}
  \label{sec:graph}
\subsection{Graph \texorpdfstring{$C^*$}{C*}-algebras}
We start out by recalling the definition of the $C^*$-algebra associated to a directed graph \cite{FLR00}, together with results about its K-theory and ideal structure.

A directed graph $E=(E^0,E^1,r,s)$ consists of a countable set $E^0$ of \textit{vertices}, a countable set $E^1$ of \textit{edges} and two maps $r,s: E^1\to E^0$ called the \textit{range map} and the \textit{source map} respectively.  For an edge $e\in E^1$ from $v$ to $w$ we have $s(e)=v$ and $r(e)=w$. 
A \textit{path} $\alpha$ in a graph is a finite sequence $\alpha=e_1e_2\cdots e_n$ of edges satisfying $r(e_i)=s(e_{i+1})$ for $i=1,...,n-1$. We denote by $E^*$ all paths of finite length in the graph $E$. 

A vertex $v\in E^0$ is called \textit{regular} if the set 
$s^{-1}(v):=\{e\in E^1| \ s(e)=v\}$
is finite and non-empty.
A vertex $v$ is called a \textit{sink} if it emits no edges i.e. $s^{-1}(v)$ is empty. A graph $E$ is \textit{row-finite} if every vertex in $E^0$ is either regular or a sink.

\begin{defn} Let $E=(E^0,E^1,r,s)$ be a directed graph. The graph $C^*$-algebra $C^*(E)$
is the universal $C^*$-algebra generated by families of projections $\{P_v | \ v\in E^0\}$ and partial isometries $\{S_e | \ e\in E^1\}$ satisfying, for all $v, w \in E^0$ and $e, f \in E^1$, the relations
\begin{itemize}
\item[(i)] $P_vP_w=0$, for $v\neq w$;
\item[(ii)] $S_e^*S_f=0$, for $e\neq f$;
\item[(iii)] $S_e^*S_e=P_{r(e)}$;
\item[(iv)] $S_eS_e^*\leq P_{s(e)}$;
\item[(v)] $P_v=\underset{s(e)=v}{\sum}S_e S_e^*$, for every $v \in E^0$ regular.
\end{itemize}
Conditions (iii)-(v) are known as the \textit{Cuntz--Krieger relations}. 
Note that relation (iv) is equivalent to $P_{s(e)}S_eS_e^*=S_eS_e^*$, see \cite[Chapter 5]{R05}.
\end{defn}
By universality, we can define a circle action $\gamma: \mathbb{T}\to \text{Aut}(C^*(E))$, called the \textit{gauge action}, for which 
$$\gamma_z(P_v)=P_v \ \text{and} \ \gamma_z(S_e)=zS_e$$ 
for all $v\in E^0, e\in E^1$ and $z\in \mathbb{T}$. 

\subsubsection{$K$-theory} The $K$-theory of graph $C^*$-algebras has over the time been been described under various assumptions on the graph: first for Cuntz--Krieger algebras in \cite{PR96}, then in the case of row-finite graphs \cite{RS04}. We present here the description of the $K$-theory groups for a general graph $E$ from \cite{DT02}.

Let $V_E \subseteq E^0$ denote the collection of all the regular vertices. Let $\mathbb{Z} V_E$ and $\mathbb{Z} E^0$ be the free abelian groups on free generators $V_E$ and $E^0$, respectively. We define a map $K_E: \mathbb{Z} V_E\to \mathbb{Z} E^0$ as follows
\begin{equation}
    \label{eq:kercoker}
K_E(v)=\left(\sum_{e\in E^1: \ s(e)=v} r(e)\right)-v.
\end{equation}
Then \cite[Theorem 3.1]{DT02} (see also  \cite[Proposition 2]{S02})  yields
\begin{equation}
    \label{eq:mapKth}
K_0(C^*(E))\cong \text{coker}(K_E), \qquad \ K_1(C^*(E))\cong \ker(K_E). 
\end{equation}
If $E$ is a row finite graph with no sinks, the above corresponds to taking the cokernel and the kernel of $A_E^T-1$, where $A_E$ is the adjacency matrix of the graph, see \cite[Theorem 3.2]{RS04}. 

\subsubsection{Gauge-invariant ideals}
\label{sub:gauge}
The ideal structure of a graph \Cs algebra can also be read off from the underlying graph. 
We shall now describe the gauge-invariant ideals of $C^*(E)$ which arise from hereditary and saturated subsets. A subset $H\subseteq E^0$ is called \textit{hereditary} if the following condition is satisfied: If $v\in H$ and $w\in E^0$ is such that there exists a path from $v$ to $w$ then $w\in H$. A subset $S\subseteq E^0$ is called \textit{saturated} if the following condition is satisfied: If $w$ is a regular vertex in $E^0$ and for each $e\in E^1$, for which $s(e)=w$, we have $r(e)\in S$, then $w\in S$. In other words, if all the outgoing edges from $w$ end inside $S$ then $w$ is also in $S$.

It was shown in \cite{BPRS00} that gauge-invariant ideals of the \Cs algebra correspond to hereditary and saturated subsets of the vertex set.
Let $\Sigma_E$ denote the collection of all hereditary and saturated subset $H\subseteq E^0$. For each $H\in\Sigma_E$ we obtain a gauge-invariant ideal: the ideal generated by $\{P_v| \ v\in H\}$. We denote this ideal by $I_H$. 
Given a row-finite graph $E$, \cite[Theorem 4.1]{BPRS00} establishes a one-to-one correspondence between $\Sigma_E$ and the gauge-invariant ideals of $C^*(E)$. The correspondence is given by the following maps: 
\[
H\longmapsto I_H, \qquad J\longmapsto \{v\in E^0| \ P_v\in J\},
\]
for $H\in \Sigma_E$ and $J$ a gauge-invariant ideal. A word of caution is needed here: if some of the vertices in the graph emit infinitely many edges, not every gauge-invariant ideal need to take the form $I_H$ for a $H\in\Sigma_E$. This phenomenon is thoroughly described in \cite[Theorem 3.6]{BHRS02}, where the authors also provide a complete description of all gauge-invariant ideals of an infinite graph. 

Moreover, if $E$ is row-finite and $H\in\Sigma_E$, then the quotient \Cs algebra $C^*(E)/I_H$ is a graph algebra, isomorphic to $C^*(F)$,  where $F$ is the directed graph defined by setting
\[ F^0=E^0\setminus H, \qquad F^1:=\{e\in E^1| \ r(e)\notin H\},\] and with range and source maps obtained from the ones from the graph $E$ \cite[Theorem 4.1]{BPRS00}. 

For graph  $C^*$-algebras that are not row-finite, like the quantum complex projective spaces, one needs the more advanced description of the quotient $C^*(E)/I_H$ as a graph $C^*$-algebra from \cite{BHRS02}. Let $H\in\Sigma_E$ and define
$$
H_{\infty}^{\text{fin}}:=\left\lbrace v\in E^0\setminus H \ : \ |s^{-1}(v)|=\infty \ \text{and} \ 0<|s^{-1}(v)\cap r^{-1}(E^0\setminus H)|<\infty\right\rbrace.
$$
Let $E/H$ be the directed graph for which 
$$
\begin{aligned}
&(E/H)^0=(E^0\setminus H)\cup \{\beta(v)| \ v\in H_{\infty}^{\text{fin}}\}, \\
&(E/H)^1=r^{-1}(E^0\setminus H)\cup \{\beta(e)| \ e\in E^1, r(e)\in H_{\infty}^{\text{fin}}\}, 
\end{aligned}
$$
where the symbols $\beta(v)$ and $\beta(e)$ denote the vertices and edges which have been added to the graph $F$ from before. Note that all $\beta(v)$ will be sinks. The range and the source maps are extended from $E$ by setting $s(\beta(e))=s(e)$ and $r(\beta(e))=\beta(r(e))$.
If $E$ is row-finite, then $H_{\infty}^{\text{fin}}=\emptyset$, and we get $F=E/H$, as above.
By \cite[Corollary 3.5]{BHRS02} we have that $C^*(E)/I_H$ is isomorphic to $C^*(E/H)$. 

This has the important consequence that for any $H\in\Sigma_E$, one gets a short exact sequence of $C^*$-algebras: 
$$
\xymatrix{0 \ar[r] & I_H \ar[r]& C^*(E) \ar[r]& C^*(E/H)\ar[r]& 0.}
$$
As we will describe in the next Section, exactness of the sequence \eqref{exactseqn} follows from considerations of this kind.

\section{Quantum complex projective spaces}
\label{sec:QPS}
We will now introduce our main object of study, namely the \Cs algebras of quantum projective spaces, and describe their K-theory and ideal structure.

For $q\in (0,1)$, the quantum $(2n+1)$-sphere $C(S_q^{2n+1})$ of Vaksman and Soibelman \cite{VS91} is defined as universal $C^*$-algebra generated by $z_0,z_1,...,z_n$ subject to the following relations: 
\begin{equation}
    \label{eq:qps}
\begin{aligned}
&z_iz_j=q^{-1}z_jz_i, \;\;\; \text{for} \; i<j, \ \ z_iz_j^*={q^{-1}}z_j^*z_i, \;\;\; \text{for} \; i\neq j, \\
&z_i^*z_i=z_iz_i^*+(1-q^{2})\sum_{j=i+1}^n z_jz_j^*,  \;\;\; \text{for} \; i=0,\dotsc, n, \\ &\sum_{j=0}^n z_jz_j^*=1.
\end{aligned} 
\end{equation}
The complex projective space $C(\CP_q^n)$ is obtained as the fixed point algebra under the circle action on $C(S^{2n+1}_q)$  given on generators by $z_i\mapsto wz_i, w\in U(1)$, and extended by universality. The fixed point algebra is generated by elements $p_{ij}:=z_i^*z_j$ for $i,j=0,1,\ldots,n$, which satisfy commutation relations that can be obtained from those of the quantum sphere $C(S^{2n+1}_q)$:
\begin{align}\label{rel:qps}
p_{ij}p_{kl}&=q^{\mathrm{sign}(k-i)+\mathrm{sign}(j-l)}\, p_{kl}p_{ij}
 &\hspace{-1.5cm}\mathrm{if}\;i\neq l,\;j\neq k\;, \nn \\
p_{ij}p_{jk}&=q^{\mathrm{sign}(j-i)+\mathrm{sign}(j-k)+1}\, p_{jk}p_{ij}-(1-q^2)
\textstyle{\sum_{l>j}}\, p_{il}p_{lk} &\mathrm{if}\;i\neq k\;, \nn \\
p_{ij}p_{ji}&=
q^{2\mathrm{sign}(j-i)}p_{ji}p_{ij}+
(1-q^2)\left(\textstyle{\sum_{l>i}}\,q^{2\mathrm{sign}(j-i)}p_{jl}p_{lj}
-\textstyle{\sum_{l>j}}\,p_{il}p_{li}\right)
 &\mathrm{if}\;i\neq j,
\end{align}
with $\mathrm{sign}(0):=0$.
The elements $p_{ij}$ are  the matrix entries of an $(n+1)\times(n+1)$ projection $P= (p_{ij})$, and satisfy
$\sum_{j=0}^n p_{ij} p_{jk} = p_{ik}$ and $p_{ij}^*=p_{ji}$. 

\subsection{Quantum sphere and complex projective spaces as graph \texorpdfstring{$C^{*}$}{C*}-algebras}\label{CPGraphAlgebra}
In this subsection, we will recall the main results from \cite{HS02} and describe how odd quantum spheres and projective spaces can be studied within the framework of graph $C^{*}$-algebras. Let $L_{2n+1}$ be the directed graph with $n+1$ vertices, denoted $\{v_1,v_2,...,v_{n+1}\}$ and, for each $i\leq j$, a single edge $e_{ij}$ from $v_i$ to $v_j$. 

As an example, if $n=3$, the graph $L_7$ will be as follows. 
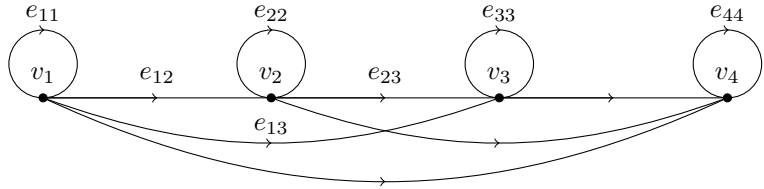
\begin{center}
\begin{tikzpicture}[scale=1.5]
\draw (0,0) -- (6,0);
\filldraw [black] (0,0) circle (1pt);
\filldraw [black] (2,0) circle (1pt);
\filldraw [black] (4,0) circle (1pt);
\filldraw [black] (6,0) circle (1pt);
\draw [->] (0,0) -- (1,0);
\draw [->] (2,0) -- (3,0);
\draw [->] (4,0) -- (5,0);

\draw[] (0,0) to [out=-20,in=-160] (4,0);
\draw [->] (1.99,-0.4) -- (2,-0.4);
\draw[] (2,0) to [out=-20,in=-160] (6,0);
\draw [->] (3.99,-0.4) -- (4,-0.4);
\draw[] (0,0) to [out=-25,in=-155] (6,0);
\draw [->] (2.99,-0.74) -- (3,-0.74);

\draw (0,0.3) circle [radius=0.3cm];
\draw (2,0.3) circle [radius=0.3cm];
\draw (4,0.3) circle [radius=0.3cm];
\draw (6,0.3) circle [radius=0.3cm];
\draw [->] (-0.01,0.6) -- (0.01,0.6);
\draw [->] (1.99,0.6) -- (2.01,0.6);
\draw [->] (3.99,0.6) -- (4.01,0.6);
\draw [->] (5.99,0.6) -- (6.01,0.6);

\node at (0, 0.2)  {$v_1$};
\node at (2, 0.2)  {$v_2$};
\node at (4, 0.2)  {$v_3$};
\node at (6, 0.2)  {$v_4$};
\node at (1, 0.2)  {$e_{12}$};
\node at (3, 0.2)  {$e_{23}$};
\node at (2, -0.25)  {$e_{13}$};
\node at (0, 0.75)  {$e_{11}$};
\node at (2, 0.75)  {$e_{22}$};
\node at (4, 0.75)  {$e_{33}$};
\node at (6, 0.75)  {$e_{44}$};

\end{tikzpicture}
\end{center}
It follows from \cite[Theorem 4.4]{HS02} that $C(S_q^{2n+1})$  is isomorphic to the graph $C^*$-algebra $C^*(L_{2n+1})$.

Under the isomorphism of $C(S_q^{2n+1})$ and $C^*(L_{2n+1})$, the $U(1)$-action on $C(S_q^{2n+1})$ defining $C(\CP_q^n)$ becomes 
$$
S_{e_{ij}}\mapsto 
wS_{e_{ij}}  \ \ \ P_{v_i}\mapsto P_{v_i}, \ w\in U(1),
$$
which is precisely the gauge action, $\gamma$, on the graph $C^*$-algebra $C^*(L_{2n+1})$. 

In order to realize $C(\CP_q^n)$ as a graph \Cs algebra, consider the directed graph $F_n$ with vertices $\{w_1,...,w_{n+1}\}$ and infinitely many edges from $w_i$ to $w_j$ if $i<j$ for $i,j=1,...,{n+1}$. 

As an example, if $n=3$, the graph $F_3$ will be as follows. 
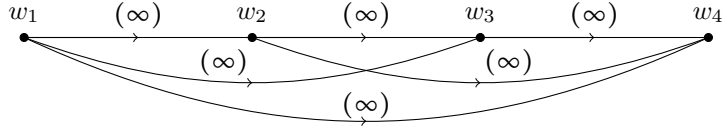
\begin{center}
\begin{tikzpicture}[scale=1.5]
\draw (0,0) -- (6,0);
\filldraw [black] (0,0) circle (1pt);
\filldraw [black] (2,0) circle (1pt);
\filldraw [black] (4,0) circle (1pt);
\filldraw [black] (6,0) circle (1pt);
\draw [->] (0,0) -- (1,0);
\draw [->] (2,0) -- (3,0);
\draw [->] (4,0) -- (5,0);

\draw[] (0,0) to [out=-20,in=-160] (4,0);
\draw [->] (1.99,-0.4) -- (2,-0.4);
\draw[] (2,0) to [out=-20,in=-160] (6,0);
\draw [->] (3.99,-0.4) -- (4,-0.4);
\draw[] (0,0) to [out=-25,in=-155] (6,0);
\draw [->] (2.99,-0.74) -- (3,-0.74);

\node at (0, 0.2)  {$w_1$};
\node at (2, 0.2)  {$w_2$};
\node at (4, 0.2)  {$w_3$};
\node at (6, 0.2)  {$w_4$};
\node at (1, 0.2)  {$(\infty)$};
\node at (3, 0.2)  {$(\infty)$};
\node at (5, 0.2)  {$(\infty)$};
\node at (1.75, -0.22)  {$(\infty)$};
\node at (4.25, -0.22)  {$(\infty)$};
\node at (3, -0.6)  {$(\infty)$};

\end{tikzpicture}
\end{center}

\noindent Then, as stated in \cite{HS02}, we have 
$$
C(\CP_q^n)\cong C^*(L_{2n+1})^{\gamma}\cong C^*(F_n).
$$
The proof essentially relies on considering all the paths $\alpha$ and $\beta$ in the graph $L_{2n+1}$ such that $S_{\alpha}S_{\beta}^*$ is invariant under the gauge action $\gamma$. 

\subsubsection{Representations of $C(S_q^{2n+1})$ and $C(\CP_q^n)$}
An irreducible representation of the Vaksman--Soibelman $C(S_q^{2n+1})$ is obtained in \cite[Section 5.4]{La05} as follows:
\begin{equation}
\label{eq:repVS}
\psi: C(S_q^{2n+1})\to B(l^2(\nn^n))
\end{equation}
given on the generators by:
$$
\begin{aligned}
\psi(z_0)\zeta(k_1,...,k_n)&=\sqrt{1-q^{2(k_1+1)}}\zeta(k_1+1,...,k_n), \\
\psi(z_j)\zeta(k_1,...,k_n)&=q^{k_1+\cdots +k_{j}}\sqrt{1-q^{2(k_{j+1}+1)}} \zeta(k_1,...,k_{j},k_{j+1}+1,k_{j+2},...,k_n), \\
\psi(z_n)\zeta(k_1,...,k_n)&=q^{k_1+\cdots + k_{n}}\zeta(k_1,...,k_n),
\end{aligned}
$$
for $j=1,...,n$ and $k_1,...,k_{n}\in\nn$. 
\begin{rem}
The representation $\psi$ is obtained from the representation $\psi_1^{2n+1}$ from \cite{La05} by replacing $q$ with $q^{-1}$, renaming the generators $x_i$ by $z_{n-i+1}^*$ for $i=1,...,n+1$ and in the end change the basis such that $k_i$ is replaced by $k_{n-i+1}$. 
\end{rem}
From the representation $\psi$ we can obtain the faithful representation $\pi: C(S_q^{2n+1})\to B(l^2(\nn^n\times\zz))$, as defined in \cite[Lemma 4.1]{HS02}, given on the generators by: 
$$
\begin{aligned}
\pi(z_0)\zeta(k_1,...,k_n,m)&=\sqrt{1-q^{2(k_1+1)}}\zeta(k_1+1,...,k_n.m), \\
\pi(z_j)\zeta(k_1,...,k_n,m)&=q^{k_1+\cdots +k_{j}}\sqrt{1-q^{2(k_{j+1}+1)}} \zeta(k_1,...,k_{j},k_{j+1}+1,k_{j+2},...,k_n,m), \\
\pi(z_n)\zeta(k_1,...,k_n,m)&=q^{k_1+\cdots + k_{n}}\zeta(k_1,...,k_n,m+1),
\end{aligned}
$$
for $j=1,...,n$, $k_1,...,k_{n}\in\nn$ and $m\in\zz$.
\\

A faithful representation of the corresponding graph $C^*$-algebra $C^*(L_{2n+1})$ is given in \cite{HS03} by\[
\rho: C^*(L_{2n+1})\to B(l^2(\nn^n\times \zz))
\]
such that 
\begin{equation}\label{RepRho}
\begin{aligned}
\rho(P_{v_{n+1}})\xi(k_1,...,k_n,m)&= \delta_{k_1,0}\cdots\delta_{k_{n},0}\xi(k_1,...,k_n,m), \\
\rho(P_{v_j})\xi(k_1,...,k_n,m)&=\delta_{k_1,0}\cdots \delta_{k_j,0}(1-\delta_{k_{j+1},0})\xi(k_1,...,k_n,m), \\
\rho(S_{e_{n+1,n+1}})\xi(k_1,...,k_n,m)&=\delta_{k_1,0}\cdots \delta_{k_n,0}\xi(k_1,...,k_n,m+1), \\
\rho(S_{e_{j,n+1}})\xi(k_1,...,k_n,m)&=\delta_{k_1,0}\cdots\delta_{k_n,0}\xi(k_1,...,k_j,k_{j+1}+1,k_{j+2},...,k_n,m), \\
\rho(S_{i,j})\xi(k_1,...,k_n,m)&=\delta_{k_1,0}\cdots \delta_{k_{j},0}(1-\delta_{k_{j+1},0})\xi(k_1,...,k_i,k_{i+1}+1,k_{i+2},...,k_n,m),
\end{aligned}
\end{equation}
for $j=1,...,n, i=1,...j, m\in \zz$ and $k_1,...,k_n\in \nn$. Here $\delta$ is the Kronecker symbol. 

When working with the isomorphism $C(S_q^{2n+1})\cong C^*(L_{2n+1})$, the following result will become useful: \begin{prop}{\cite[Remark 4.5]{HS03}}
The map $\rho^{-1}\circ \pi$ implements an isomorphism between $C(S_q^{2n+1})$ and $C^*(L_{2n+1})$.
\end{prop}
In \cite{HS02} Hong and Szyma\'nski provide and explicit form of the isomorphism.

We then obtain faithful representations of $C(\CP_q^n)$ and $C^*(F_n)$ by restricting the representations $\pi$ and $\rho$ respectively. Later in Section~\ref{ProjectionsFredholmModules}, we will show how these two representations relate to each other and use this fact to construct a different basis of generators for the $K$-theory of quantum projective spaces. 

\subsection{Ideal structure and extensions}
We will now exploit the description of the ideal structure of a graph \Cs algebra in terms of hereditary and saturated subsets presented in Subsection~\ref{sub:gauge}, to obtain the $C^*$-algebra extension \eqref{exactseqn}.

In $F_n$, we consider the hereditary and saturated subset $H:=\{w_{n+1}\}$. In this case $H_{\infty}^{\text{fin}}=\emptyset$ since $w_{n+1}$ is a sink. Then $C^*(F_n)/I_H\cong C^*(F_n/\{w_{n+1}\})$ which is $C^*(F_{n-1})$. The ideal $I_{\{w_{n+1}\}}$ is isomorphic to $\mathcal{K}$. 
Indeed, by \cite[Lemma 3.2]{BHRS02} we have 
$$I_{\{ w_{n+1}\}}=\overline{\text{span}}\{s_\alpha s_\beta^*| \ \alpha,\beta\in E^*, \ r(\alpha)=r(\beta)=w_{n+1}\}.$$
It can be shown that $\{f_{\alpha,\beta}:=s_{\alpha}s_{\beta}^*| \alpha,\beta\in E^*, r(\alpha)=r(\beta)=w_{n+1}\}$ forms a set of matrix units $I_{\{w_{n+1}\}}$. Hence $I_{\{w_{n+1}\}}\cong \mathcal{K}(l^2(\{\alpha\in E^*| \ r(\alpha)=w_{n+1}\}))$.

We then obtain a short exact sequence:
\begin{equation}\label{splitting}
\xymatrix{0 \ar[r] & \cK \ar[r]_-{j_n} & C(\CP^n_q) \ar[r]_{q_n}  & C(\CP^{n-1}_q) \ar[r] & 0}.
\end{equation}

In Section \ref{SplittingProjectivespace} we will prove that the exact sequence is split exact. This is a crucial step in our construction of an explicit $KK$-equivalence between $C(\CP_q^{n})$ and $\mathbb{C}^{n+1}$. 

\subsection{K-theory and K-homology of quantum projective spaces} 

As mentioned in the introduction, the $K$-theory groups for the \Cs algebras of quantum projective spaces $C(\CP^n_q)$ are given by
\[K_0(C(\CP_q^n))\cong \mathbb{Z}^{n+1}, \qquad  K_1(C(\CP_q^n))\cong \{0\}.\]

This fact can be proved by viewing the $C^*$-algebra $C(\CP^n_q)$ as the graph $C^*$-algebra $C^*(F_n)$. Since $V_{F_{n}}=\emptyset$, using \eqref{eq:mapKth} and \eqref{eq:kercoker}, we obtain $K_{F_{n}}:\{0\}\to \mathbb{Z}^{n+1}$ which has cokernel $\mathbb{Z}^{n+1}$ and the kernel is $0$.

The dual result for the $K$-homology of quantum projective spaces is obtained similarly and leads to   \[K^0(C(\CP^n_q))\simeq\mathbb{Z}^{n+1}, \qquad K^1(C(\CP^n_q))\simeq \lbrace 0 \rbrace.\]

\begin{rem}
\label{rem:comm}
One should compare those results with their analogs in the commutative case, where one also has a $KK$-equivalence between the algebras $C(\CP^{n})$ and $\mathbb{C}^{n+1}$.  Let $\CP^{n-1}$ and $\CP^{n}$ denote the complex projective space of $\mathbb{C}^{n}$ and $\mathbb{C}^{n+1}$, respectively.  
Since $\CP^{n-1}$ is a closed subspace of the compact topological space $\CP^{n}$, the corresponding \Cs algebras of continuous functions fit into an extension of the form
\begin{equation}\label{eq:comm_ext}
\xymatrix{0 \ar[r] & C_0 (\mathbb{C}^{n}) \ar[r] & C(\CP^n) \ar[r]  & C(\CP^{n-1}) \ar[r] & 0},
\end{equation}
which induces a corresponding six-term exact sequence in $K$-theory generalizing the relative K-theory exact sequence in topological $K$-theory \cite[Corollary~II.3.23]{Ka78}. It follows that the K-groups of all $C(\CP^n)$ can be computed inductively, obtaining that they are equal to those of $\mathbb{C}^{n+1}$. A crucial step in the computation is the observation that by Bott periodicity $K_i(C_0 (\mathbb{C}^{n})) \simeq K_i(\mathbb{C})$ for $i=0,1$. Here again, the $KK$-equivalence follows from the fact that all commutative \Cs algebras are in the UCT class.
\end{rem}

As described in the introduction $C(\CP_q^n)$ is in the UCT class and hence $C(\CP_q^n)$ is $KK$-equivalent to $\mathbb{C}^{n+1}$. Remark~\ref{rem:comm} then implies that it is also $KK$-equivalent to its commutative counterpart $C(\CP^n)$.

\section{A splitting for the defining extension of quantum projective spaces}\label{SplittingProjectivespace}

In this section, we construct a splitting of the exact sequence \eqref{splitting} from which we obtain an explicit $KK$-equivalence between $C(\CP_q^{n})$ and $C(\CP_q^{n-1})\oplus \cK$. By induction, and up to Morita equivalence, this will allow us to obtain the desired $KK$-equivalence between $C(\CP_q^{n})$ and $\mathbb{C}^{n+1}$.

As mentioned in the Introduction, the existence of such a splitting is a direct consequence of the Universal Coefficient Theorem \cite{RoSch87}. While this observation is certainly well-known to the experts, we restate it here for the sake of completeness. 

\begin{lem}
\label{lem:split}
Let $A$ be a separable \Cs algebra in the UCT-class, with $K_0(A)$ free abelian and vanishing $K_1(A)$. Let $m \geq 1$ and denote by $\mathcal{K}$ the algebra of compact operators. Then any extension of $A$ by $\mathcal{K}^{\oplus m}$ splits.
\end{lem}
\begin{proof}
Extensions of the form \[\xymatrix{0 \ar[r] & \cK^{\oplus m} \ar[r] & E \ar[r] & A  \ar[r] & 0},
 \]
 are classified by the Kasparov group $KK^1(A, \mathcal{K}^{\oplus m}) \simeq KK^1(A, \mathcal{K})^{\oplus m} $, which we can describe in terms of the K-groups of $A$ thanks to the UCT and the Morita equivalence between $\mathcal{K}$ and $\mathbb{C}$. 
 
 Since $K_0(A)$ is a free abelian group, the group $\mathrm{Ext}^1_{\mathbb{Z}}(K_0(A),\mathbb{Z})$ vanishes, yielding
  \[ KK_1(A, \mathcal{K}^m) \simeq \mathrm{Hom}_{\mathbb{Z}} (K_1(A), \mathbb{Z})^{\oplus m} \oplus \mathrm{Hom}_{\mathbb{Z}} (K_0(A), \lbrace 0 \rbrace)^{\oplus m} \simeq \lbrace 0 \rbrace,\]
by virtue of our assumption on $K_1(A)$. It follows that there are no non-trivial extension of the form above, that is, all such extensions must necessarily split.
\end{proof}
As a consequence,  all the defining extensions for quantum projective spaces split.  This is also true for a class of weighted projective spaces satisfying a suitable assumption on the weight vector, like those studied in \cite{BS16} and \cite{A18}, but we shall postpone the treatment of that case to later work, as the question regarding which graphs underlie such algebras has not been settled yet. It is worth noting that an explicit $KK$-equivalence in the one-dimensional case, that is for quantum teardrops, can be found in \cite{AKL16}. 

We shall now describe our splitting explicitly in the graph algebra picture. Since we have to consider two complex projective spaces at once, we will denote vertices and edges of their graphs 
with different letters. 
\begin{figure}[H]

\begin{center}
\begin{tikzpicture}[scale=1.5]
\draw (0,0) -- (2,0);
\draw  [dashed] (2,0) -- (4,0);
\draw (4,0) -- (6,0);
\filldraw [black] (0,0) circle (1pt);
\filldraw [black] (2,0) circle (1pt);
\filldraw [black] (4,0) circle (1pt);
\filldraw [black] (6,0) circle (1pt);
\draw [->] (0,0) -- (1,0);
\draw [->] (4,0) -- (5,0);

\draw[] (0,0) to [out=20,in=160] (4,0);
\draw [->] (1.99,0.4) -- (2,0.4);
\draw[] (2,0) to [out=20,in=160] (6,0);
\draw [->] (3.99,0.4) -- (4,0.4);
\draw[] (0,0) to [out=-25,in=-155] (6,0);
\draw [->] (2.99,-0.74) -- (3,-0.74);

\node at (-1, 0)  {$F_{n}$};
\node at (0, 0.2)  {$w_1$};
\node at (2, 0.2)  {$w_2$};
\node at (4, 0.2)  {$w_{n}$};
\node at (6, 0.2)  {$w_{n+1}$};
\node at (1.4, -0.2)  {$(\infty)$ $f_{12}^m$};
\node at (4.5, -0.2)  {$(\infty)$ $f_{n(n+1)}^m$};
\node at (1.75, 0.65)  {$(\infty)$ $f_{1n}^m$};
\node at (4.25, 0.65)  {$(\infty)$ $f_{2(n+1)}^m$};
\node at (3, -0.6)  {$(\infty)$};
\node at (3, -1)  {$f_{1(n+1)}^m$};
\end{tikzpicture}
\end{center}
\caption{The graph $F_n$ such that $C(\CP_q^n)\cong C^*(F_n)$. The symbol $(\infty)$ indicates that there are infinitely many edges between the vertices. 
}
\label{Fig:Graphs}
\end{figure}
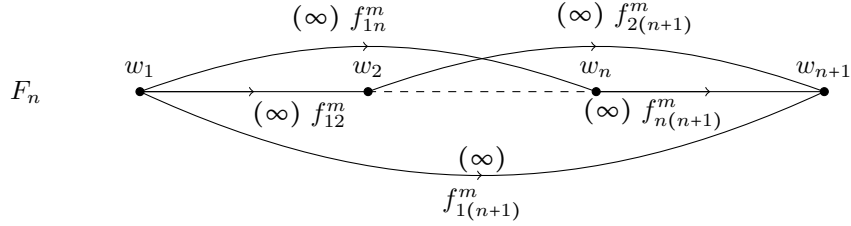

We label the vertices of $F_n$ with $w$ and the edges with $f$, like in Figure~\ref{Fig:Graphs}, while for the graph $F_{n-1}$ we chose the label $v_i$ for $i=1,...,n$ for the vertices and $ e_{ij}^{m}$ for  $1\leq i< j \leq n $, $m\in\mathbb{N}$ for the edges.
Then 
$C(\CP_q^{n-1})$ is isomorphic to the universal $C^*$-algebra $C^*(F_{n-1})$, generated by projections $P_{v_i},i=1,...,n$ and partial isometries $S_{e_{ij}^m}$ for  $1\leq i< j \leq n $, $m\in\mathbb{N}$, subject to the relations 
\begin{align}
P_{v_i}P_{v_j}&=0, && i\neq j, \label{relation1} \\
S_{e_{ij}^m}^*S_{e_{kl}^m}&=0, && (i,j)\neq (k,l), \label{relation2}
\\
S_{e_{ik}^m}^*S_{e_{ik}^m}&=P_{v_k}, && k=2,...,n, i=1,...,k-1, \label{relation3} \\
S_{e_{ki}^m}S_{e_{ki}^m}^*&\leq P_{v_k}, && k=1,...,n-1, i=k+1,...,n. \label{relation4}
\end{align}
Similarly, 
$C(\CP_q^{n})$ is isomorphic to the universal $C^*$-algebra, $C^*(F_{n})$, generated by projections $P_{w_i},i=1,...,n+1$ and partial isometries $S_{f_{ij}^m},1\leq i< j, j=2,...,n+1$, $m\in\mathbb{N}$ subject to the relations 

\begin{align*}
P_{w_i}P_{w_j}&=0, && i\neq j, \\
S_{f_{ij}^m}^*S_{f_{kl}^m}&=0, && (i,j)\neq (k,l),
\\
S_{f_{ik}^m}^*S_{f_{ik}^m}&=P_{w_k}, && k=2,...,n+1, i=1,...,k-1, \\
S_{f_{ki}^m}S_{f_{ki}^m}^*&\leq P_{w_k}, && k=1,...,n, i=k+1,...,n+1. 
\end{align*}

Let us now look at the exact sequence in \eqref{splitting}, which we know to be split exact by virtue of Lemma~\ref{lem:split}.  In this setting, and following the convention described above, the  quotient map $q_n$ is given by 
\begin{align*}
P_{w_{n+1}} &\mapsto 0, && \\
P_{w_i} &\mapsto P_{v_i}, && 1 \leq i \leq n,\\
S_{f^m_{i,n+1}} &\mapsto 0, && 1 \leq i \leq n,\\
S_{f^m_{i,j}} &\mapsto S_{e^m_{i,j}}, && 1 \leq i < j \leq n.
\end{align*}
\begin{thm}
The map $s_n:C(\CP_q^{n-1})\to C(\CP_q^{n})$ defined on generators by 
\begin{align*}
P_{v_i}&\mapsto P_{w_i}, && i=1,2,...,n-1, \\
P_{v_n}&\mapsto P_{w_n}+P_{w_{n+1}}, && \\
S_{e_{ij}^m}&\mapsto S_{f_{ij}^m}, && j\neq n, \\
S_{e_{i,n}^m}&\mapsto S_{f_{i,n}^m}+S_{f_{i,n+1}^m}, && i=1,...,n-1
\end{align*}

is a splitting for the short exact sequence in \eqref{splitting} 
\end{thm}
\begin{proof}
To prove that $s_n$ is a $*$-homomorphism we will show that its target elements satisfy the graph algebra relations \eqref{relation1}-\eqref{relation4}. It will then follow by universality that $s_n$ is a $*$-homomorphism. 

First, we clearly have that $P_{v_i}$ are all mapped to mutually orthogonal projections, that the unit of $C^*(F_{n-1})$ is mapped to the unit of $C^*(F_n)$, and that relation (\ref{relation2}) is satisfied. It is also clear that the partial isometries $S_{e_{ij}^m}$ are all mapped to partial isometries in $C^*(F_n)$ when $j\neq n$. When $j=n$ we have that  $S_{f_{i,n}^m}+S_{f_{i,n+1}^m}$ is indeed a partial isometry since 
$$
\begin{aligned}
&(S_{f_{i,n}^m}+S_{f_{i,n+1}^m})(S_{f_{i,n}^m}+S_{f_{i,n+1}^m})^*(S_{f_{i,n}^m}+S_{f_{i,n+1}^m})
\\
&= (S_{f_{i,n}^m}+S_{f_{i,n+1}^m})(S_{f_{i,n}^m}^*S_{f_{i,n}^m}+S_{f_{i,n+1}^m}^*S_{f_{i,n+1}^m}) \\
&= (S_{f_{i,n}^m}+S_{f_{i,n+1}^m})(P_{w_n}+P_{w_{n+1}}) \\
&= (S_{f_{i,n}^m}P_{w_n}+S_{f_{i,n+1}^m}P_{w_{n+1}})(P_{w_n}+P_{w_{n+1}}) \\
&=(S_{f_{i,n}^m}+S_{f_{i,n+1}^m}).
\end{aligned}
$$
Relation (\ref{relation3}) is clearly satisfied for all $S_{f_{ij}^m}$ with $j\neq n$. For $j=n$ we have 
$$
\begin{aligned}
S_{e_{in}^m}^*S_{e_{in}^m}\mapsto (S_{f_{i,n}^m}+S_{f_{i,n+1}^m})^*(S_{f_{i,n}^m}+S_{f_{i,n+1}^m})=P_{w_n}+P_{w_{n+1}}\leftmapsto P_{v_n}
\end{aligned}
$$
and relation (\ref{relation3}) is then satisfied in this case. 

Relation(\ref{relation4}) is also clearly obtained for all $S_{f_{ij}^m}$ with $j\neq n$. For $j=n$ we have
$$
P_{w_i}\leftmapsto P_{v_i}\geq S_{e_{in}^m}S_{e_{in}^m}^*\mapsto (S_{f_{i,n}^m}+S_{f_{i,n+1}^m})(S_{f_{i,n}^m}+S_{f_{i,n+1}^m})^*.
$$
Hence we have to show
$$
B_i:=(S_{f_{i,n}^m}+S_{f_{i,n+1}^m})(S_{f_{i,n}^m}+S_{f_{i,n+1}^m})^* \leq P_{w_i}
$$
which is equivalent to proving that $P_{w_i}B_i=B_i$. We have 
$$
P_{w_i}B_i=P_{w_i}(S_{f_{i,n}^m}S_{f_{i,n}^m}^*S_{f_{i,n}^m}+S_{f_{i,n+1}^m}S_{f_{i,n+1}^m}^*S_{f_{i,n+1}^m})(S_{f_{i,n}^m}+S_{f_{i,n+1}^m})^*=B_i
$$
since $P_{w_i}S_{f_{i,n}^m}S_{f_{i,n}^m}^*=S_{f_{i,n}^m}S_{f_{i,n}^m}^*$ and $P_{w_i} S_{f_{i,n+1}^m}S_{f_{i,n+1}^m}^*= S_{f_{i,n+1}^m}S_{f_{i,n+1}^m}^* $. 

It follows from an easy computation that $q_n\circ s_n$ is the identity on the generators of $C^*(F_{n-1})$ and therefore $q_n\circ s_n=\id_{C^*(F_{n-1})}$. 
\end{proof}
To summarize our result, for every $n\geq 1$, we have a split exact sequence
\begin{equation*}\label{splitexact}
\xymatrix{0 \ar[r] & K \ar[r]_{j_n\ \ \ } & C(\CP_q^n) \ar[r]_{q_n} & C(\CP_q^{n-1}) \ar[r]  \ar@/_/[l]_{s_n} & 0}.
\end{equation*}
\begin{rem}
In the rest of this work, especially in Section \ref{explicitKKequiv}, we will mostly be working with the graph $C^*$-algebra picture. We choose to identify $C(\CP_q^n)$ with $C^*(F_n)$, and use the former notation.
\end{rem}

\section{$KK$-equivalences for split exact sequences}
\label{sec:KKequiv}
We will now recall how any split extension of \Cs algebras gives a $KK$-equivalence between the algebra \emph{in the middle} and the \Cs algebraic direct sum of the other two. This relies on the following result, which in \cite{Bl98} is stated as an exercise.
\begin{thm}[{\cite[Exercise 19.9.1]{Bl98}}]\label{SEKKequiv}
For any split exact sequence of graded separable $C^*$-algebras 
$$
\xymatrix{0 \ar[r] & J  \ar[r]^{j} & E \ar[r]_{q} & B \ar[r]  \ar@/_/[l]_{s} & 0}
$$
the element $[j] \oplus [s] \in KK(J \oplus B, E) $ is a $KK$-equivalence. 
\end{thm}
An explicit inverse to the class $[j] \oplus [s] \in KK(J\oplus B, E) $ is also provided in \cite{Bl98}, through a construction known as the \emph{splitting homomorphism}.
We 
will illustrate this result using Cuntz's quasi-homomorphism picture of $KK$-theory \cite{Cu87}, in which all the involved $C^*$-algebras are assumed to be trivial graded and $\sigma$-unital. Our main references are the article \cite{Hig87} and the monograph \cite{JT91}.
\subsection{\texorpdfstring{$KK_h$}{KK}-theory}
For the sake of simplicity, we will further assume all \Cs algebras to be separable.
\begin{defn}[{\cite[Def. 4.1.1.]{JT91}}]
A $KK_h(A,B)$-cycle is a pair $(\phi_+,\phi_-)$ of $*$-homomorphisms from $A$ to $\mathcal{M}(\mathcal{K}\otimes B)$, such that 
$$
\phi_+(a)-\phi_-(a)\in \mathcal{K}\otimes B
$$
for all $a\in A$. The set of $KK_h(A,B)$-cycles will be denoted by  $\mathbb{F}(A,B)$. 
\end{defn}
Note that a pair of $*$-homomorphisms $(\phi_+,\phi_-)$ satisfying the condition above is also called a quasi-homomorphism from $A$ to $B$.

Let $\phi:A \to B$ be a $*$-homomorphism, and consider the $*$-homomorphism  \[
\label{eq:minproj}
e_{B}:B\to \mathcal{K}\otimes B, \ \ e_{B}(b)=e\otimes b
\]
where $e$ is a minimal projection in $\mathcal{K}$. Then the pair $(e_B \circ \phi, 0)$ is a $KK_h(A,B)$-cycle.

Homotopy of $KK_h$-cycles can be defined in a similar way to homotopy of Kasparov modules, see \cite[Def. 4.1.2.]{JT91}, in a manner that is compatible with homotopy of $*$-homomorphisms. Then $KK_h(A,B)$ is defined as the homotopy classes of $KK_h(A,B)$-cycles. Likewise, one can endow $KK_h(A,B)$ with the structure of an abelian group, as described in \cite[Proposition 4.1.5]{JT91}.

In \cite{Hig87}, Higson proved that the $KK_h(A,B)$ group is isomorphic to the original Kasparov group $KK_0(A,B)$, whenever $A$ and $B$ are considered as trivially graded $C^*$algebras (see also \cite[Theorem 4.1.8]{JT91}).

The KK groups are functorial. In the $KK_h$-picture, this is realized as follows.
Let $f:A\to B$ be a $*$-homomorphism. For any $C^*$-algebra $C$ we define a group homomorphism as follows
$$f^*:KK_h(B,C)\to KK_h(A,C)$$ 
$$
f^*[\phi_+,\phi_-]=[\phi_+\circ f, \phi_-\circ f]
$$
with $(\phi_+,\phi_-)\in \mathbb{F}(B,C)$. 

Functoriality in the other direction requires some little extra care. Let $g: \cK\otimes A\to \cK\otimes B$ be a quasi-unital $*$-homomorphism. Since $g$ is a quasi-unital $*$-homomorphism there exists by \cite[Corollary 1.1.15]{JT91} a strictly continuous extension $\underline{g}: \mathcal{M}(\cK\otimes A)\to \mathcal{M}(\cK\otimes B)$. 

Let $\{u_i\}$ be an approximate unit for $\cK\otimes A$. Identify $\mathcal{M}(\cK\otimes B)$ with $\mathcal{L}_{\cK\otimes B}(\cK\otimes B)$. Then for each $m\in \mathcal{M}(\cK\otimes A)$, 
$$\underline{g}(m): \cK\otimes A\to \cK\otimes B
$$
is given by 
$$
\underline{g}(m)(a)=\lim_i g(u_im)(a)=\lim_i g(mu_i)(a)
$$
for all $a\in \cK\otimes B$.
For any $C^*$-algebra $C$ a group homomorphism can now be defined as follows (cp. \cite[Lemma 4.1.11]{JT91}):
$$g_*:KK_h(C,A)\to KK_h(C,B)$$
$$
g_*[\phi_+,\phi_-]=[\underline{g}\circ \phi_+,\underline{g}\circ \phi_-]
$$
with $(\phi_+,\phi_-)\in \mathbb{F}(C,A)$. 
Note that if $g:A\to B$ is a $*$-homomorphism, then we can still construct $g_*$ by considering $\id_{\cK}\otimes g: \cK\otimes A\to \cK\otimes B$.

Last but not least, there is a bilinear pairing, known as the Kasparov product,
$$
\otimes : KK_h(A,B)\times KK_h(B,C)\to KK_h(A,C)
$$
satisfying the conditions of Theorem 4.2.1 in \cite{JT91}. The pullback and pushforward maps $f^*$ and $g_*$ are realized by taking the left and right Kasparov product with the $KK_h$-class induced by the *-homomorphism $f$ and $g$, respectively. In formulas:
\[
\begin{aligned} f^* &:= [f] \otimes_B - : KK_h(B,C)\to KK_h(A,C), \\
 g_* & :=  - \otimes_A [g]: KK_h(A,B) \to KK_h(A,C). 
 \end{aligned}\]

We recall that two separable $C^*$-algebras $A$ and $B$ are $KK$-equivalent if there exist an element $x\in KK(A,B)$ and $y\in KK(B,A)$ such that $x\otimes_B y=1_A$  and $y\otimes_Ax =1_B$.

\subsection{The splitting homomorphism}
Let us consider a split exact sequence of $C^*$-algebras 
$$
\xymatrix{0 \ar[r] & J  \ar[r]^{j} & E \ar[r]_{q} & B \ar[r]  \ar@/_/[l]_{s} & 0}
$$
We will now recall how to construct an inverse to the class $ [j] \oplus [s] \in KK_h(J\oplus B, E)$.

First of all, consider the $*$-homomorphism $e_E$ defined as in \eqref{eq:minproj}. Denote by $r_J$ the canonical map
given by
\begin{equation*}
\label{eq:canmap}
\begin{aligned}
r_J:\mathcal{M}(\mathcal{K}\otimes E)\to \mathcal{M}(\mathcal{K}\otimes J)\\
    r_{J}(T)(x):=(\id_{\mathcal{K}}\otimes {j}^{-1})(T(\id_{\mathcal{K}}\otimes j)(x))
\end{aligned}
\end{equation*}
for all $x\in\mathcal{K}\otimes J$, $T\in \mathcal{M}(\mathcal{K}\otimes E)$, where $j^{-1}$ is the inverse when we restrict to the image of $j$ see \cite[Exercise 1.1.9]{JT91}. 
The map satisfies 
$$
(\id_\cK\otimes j)(r_J(m)x)=m(\id_\cK\otimes j)(x)
$$
for all $m\in\Ms(\cK\otimes E)$ and $x\in \cK\otimes J$.

Denote by $[\pi] \in KK_h(E,J)$ the class of the quasi-homomorphism $(1, s \circ q)$, which we rewrite as 
$
[\pi]:=[(r_{J}\circ e_{E},r_{J}\circ e_{E}\circ s\circ q)]$ (see also \cite[Lemma~2.13]{Hig87}).

Then we have
\[\begin{aligned}
j_*([\pi])+s_*\circ q_*(1_E)&= 1_E, \\
(j^*+ s^*)([\pi]+q^*(1_B))& = 1_J+1_B=1_{J\oplus B}.
\end{aligned}\]
By \cite[Proposition 18.7.2.]{Bl98}, we conclude that $[j] \oplus [s]\in KK_h(J\oplus B, E)$ is a $KK$-equivalence with inverse $[\pi] \oplus [q]\in KK_h(E,J \oplus B)$.

\section{An explicit $KK$-equivalence between \texorpdfstring{$C(\CP_q^n)$}{CPn} and \texorpdfstring{$\mathbb{C}^{n+1}$}{Cnplusone}}
\label{explicitKKequiv}

We will now apply Theorem \ref{SEKKequiv} to construct an explicit $KK$-equivalence between $C(\CP_q^n)$ and $\cc^{n+1}$ up to Morita equivalence.  Our construction will be a special case of a construction that holds for special---but sufficiently general--families of extensions of \Cs algebras. 

\subsection{An inductive procedure to construct $KK$-equivalences from a family of splittings}
We start out by considering a family $A_n$ of separable \Cs algebras, with $A_0$ satisfying the assumptions in Lemma~\ref{lem:split}, together with extensions 
\begin{equation}
\label{eq:gen_ext}
\xymatrix{0 \ar[r] & \cK^{m_n}  \ar[r]^{j_n} & A_n \ar[r]_{q_n} & A_{n-1} \ar[r]  & 0}
\end{equation}
for $n \geq 1$. Note that, not only does the sequence for $n=1$ split, but we  also obtain by induction that all extensions \eqref{eq:gen_ext} for $n \geq 2$ satisfy the assumptions of Lemma~\ref{lem:split} and likewise split, i.e., we have maps $s_n: A_{n-1} \to A_n$ with $q_n \circ s_n= \id_{A_{n-1}}$.

As a consequence, we have $KK_h$-classes: \[
[j_n]\in KK_h(\cK^{m_n},A_n), \ \  [q_n]\in KK_h(A_n, A_{n-1}), \ \ [s_n]\in KK_h(A_{n-1},A_n)
\]
\[ [\pi_n ] =[(1, s_n\circ q_n)]
\in KK_h(A_n,\cK^{m_n}).
\]

It follows by Theorem~\ref{SEKKequiv} that $[j_n]\oplus [s_n]\in KK_h(\cK^{m_n} \oplus A_{n-1},A_n)$ is a $KK$-equivalence with inverse $[\pi_n]\oplus [q_n]\in KK_h(A_n,\cK\oplus A_{n-1})$. More concretely, for $n=1$ we have \begin{equation}\label{n=1}
\begin{aligned}
&[j_1]\otimes_{A_1}[\pi_1]=1_{\cK^{m_1}}, \ \ [s_1]\otimes_{A_1}[q_1]  =1_{A_0}, \\
&[\pi_1]\otimes_{\cK^{m_1}}[j_1]+ [q_1]\otimes_{\cc}[s_1]=1_{A_1}.
\end{aligned}
\end{equation}
and for $n\geq 2$ 
\begin{equation}\label{ngreater2}
\begin{aligned}
&[j_n]\otimes_{A_n}[\pi_n]=1_{\cK^{m_n}}, \ \ [s_n]\otimes_{A_n}[q_n]=1_{A_{n-1}}
\\
&[\pi_n]\otimes_{\cK^{m_n}}[j_n]+[q_n]\otimes_{A_{n-1}}[s_n]=1_{A_n}.
\end{aligned}
\end{equation}

We are now ready to announce and prove our $KK$-equivalence result. First note that if we let $I_1:= [j_1]\oplus [s_1]$ and $\Pi_1:=[\pi_1]\oplus [q_1]$  then 
$$
\begin{aligned}
I_1\otimes_{A_1} \Pi_1 =1_{\cK^{m_1} \oplus A_0} \ \
\Pi_1\otimes_{\cK^{m_1}\oplus A_0} I_1=1_{A_1},
\end{aligned}
$$
which follows directly from Theorem \ref{SEKKequiv}.

\begin{thm}
\label{thm:KK}
Consider a family of \Cs algebras $A_n, \ n \geq 0$, with extensions 
\begin{equation}
\label{eq:gen_ext2}
\xymatrix{0 \ar[r] & \cK^{m_n}  \ar[r]^{j_n} & A_n \ar[r]_{q_n} & A_{n-1} \ar[r]  & 0}.
\end{equation}Assume further that $A_0$ is separable and in the UCT-class, with $K_0(A)$ free abelian and vanishing $K_1(A)$. Then for every $n \geq 1$ the extension \ref{eq:gen_ext2} splits, with splittings $s_n: A_{n-1} \to A_n$.

Moreover,  let $S_n:=\sum_{j=1}^{n} m_n$, and define $KK$-classes
$$
\begin{aligned}
\Pi_n:=&[\pi_n]\oplus([q_n]\otimes_{A_{n-1}}[\pi_{n-1}])\oplus ([q_{n-1}\circ q_n]\otimes_{A_{n-2}} [\pi_{n-2}])\\ 
&\oplus \cdots\cdots
\oplus ([q_2\circ\cdots \circ q_{n-1}\circ q_n]\otimes_{A_1} [\pi_{1}])\oplus [q_1\circ q_2\circ\cdots \circ q_n],
\end{aligned}
$$
and 
$$
I_n:=[j_n]\oplus [s_n\circ j_{n-1}]\oplus [s_n\circ s_{n-1}\circ j_{n-2}]\oplus\cdots\cdots\oplus [s_n\circ s_{n-1}\circ\cdots\circ s_2\circ j_1]\oplus [s_n\circ s_{n-1}\circ\cdots\circ s_1].
$$
Then $\Pi_n\in KK_h(A_n, \cK^{S_n}\oplus A_0)$ implements a $KK$-equivalence with inverse $I_n\in KK_h(\cK^{S_n} \oplus A_0,A_n)$.
\end{thm}

\begin{proof}
The proof follows by induction on $n$.
For $n=2$ we obtain the following by (\ref{ngreater2}):
$$
\begin{aligned}
I_2\otimes_{A_2}\Pi_2
&= [j_2]\otimes_{A_2} [\pi_2]+[s_2\circ j_1]\otimes_{A_2}([q_2]\otimes_{A_1}[\pi_1])+[s_2\circ s_1]\otimes_{A_2}[q_1\circ q_2] \\
&= 1_{\cK^{m_2}}+ [j_1]\otimes_{A_1}([s_2]\otimes_{A_2}[q_2])\otimes_{A_1}[\pi_1]+[q_1\circ q_2\circ s_2\circ s_1] \\
&= 1_{\cK^{m_2}}+ [j_1]\otimes_{A_1} 1_{A_1}\otimes_{A_1}[\pi_1]+[q_1\circ \id_{A_1}\circ s_1] \\
&= 1_{\cK^{m_2}}+1_{\cK^{m_1}}+1_{A_{0}} \\
&=1_{\cK^{S_2}\oplus A_0}.
\end{aligned}
$$
When taking the product in the other direction we obtain 
$$
\begin{aligned}
\Pi_2&\otimes_{\cK^{m_1+m_2}\oplus A_{0}}I_2 \\
&=
[\pi_2]\otimes_{\cK^{m_2}}[j_2] +([q_2]\otimes_{A_1} [\pi_1])\otimes_{\cK^{m_2}} [s_2\circ j_1]+[q_1\circ q_2]\otimes_{A_{0}} [s_2\circ s_1] \\
&=[\pi_2]\otimes_{\cK^{m_2}} [j_2]+([q_2]\otimes_{A_1}[\pi_1])\otimes_{\cK^{m_1}} ([j_1]\otimes_{A_1}[s_2]) \\
& \hspace{0.2cm}+([q_2]\otimes_{A_1}[q_1])\otimes_{A_{0}} ([s_1]\otimes_{A_1}[s_2]) \\
&= [\pi_2]\otimes_{\cK^{m_2}}[j_2] + [q_2]\otimes_{A_1}\Big([\pi_1]\otimes_{\cK^{m_1}} [j_1]+[q_1]\otimes_{A_{0}}[s_1]\Big)\otimes_{A_1} [s_2] \\
&= [\pi_2]\otimes_{\cK^{m_2}}[j_2] + [q_2]\otimes_{A_1}1_{A_1}\otimes_{A_1} [s_2] \\
&= [\pi_2]\otimes_{\cK^{m_2}}[j_2] + [q_2]\otimes_{A_1} [s_2]\\ &=1_{A_2}.
\end{aligned}
$$
Let us assume that the statement is true for $n-1$, i.e., 
\begin{equation}\label{hypothesis}
I_{n-1}\otimes_{A_{n-1}} \Pi_{n-1}=1_{\cK^{S_{n-1}}\oplus A_{0}}, \ \
\Pi_{n-1}\otimes_{\cK^{S_{n-1}}\oplus A_{0}} I_{n-1}=1_{A_{n-1}}. 
\end{equation}
Then we can rewrite
$$
\begin{aligned}
I_{n}\otimes_{A_{n}} \Pi_{n}
&= [j_n]\otimes_{A_{n}}[\pi_n]+[s_n\circ j_{n-1}]\otimes_{A_{n}}([q_n]\otimes_{A_{n-1}}[\pi_{n-1}])\\
\ \ \ &+[s_n\circ s_{n-1}\circ j_{n-2}]\otimes_{A_{n}}([q_{n-1}\circ q_n]\otimes_{A_{n-2}}[\pi_{n-2}]) \\
\ \ \ &+\cdots +[s_n\circ s_{n-1}\circ \cdots\circ s_2 \circ j_{1}]\otimes_{A_{n}}([q_2\circ \cdots \circ q_{n-1}\circ q_n]\otimes_{A_{1}}[\pi_{1}]) \\
&+ [s_n\circ s_{n-1}\circ\cdots\circ s_1]\otimes_{A_n}[q_1\circ q_2\circ\cdots\circ q_n]  \\
&= [j_n]\otimes_{A_{n}}[\pi_n]+[j_{n-1}]\otimes_{A_{n-1}} ([s_n]\otimes_{A_{n}}[q_n])\otimes_{A_{n-1}}[\pi_{n-1}]\\
\ \ \ &+[j_{n-2}]\otimes_{A_{n-2}}([s_n\circ s_{n-1}]\otimes_{A_{n}}[q_{n-1}\circ q_n])\otimes_{A_{n-2}}[\pi_{n-2}]) \\
\ \ \ &+\cdots +[j_{1}]\otimes_{A_{1}}([s_n\circ s_{n-1}\circ \cdots\circ s_2 ]\otimes_{A_{n}}[q_2\circ \cdots \circ q_{n-1}\circ q_n])\otimes_{A_{1}}[\pi_{1}] \\
&+ [s_n\circ s_{n-1}\circ\cdots\circ s_1]\otimes_{A_{n}}[q_1\circ q_2\circ\cdots\circ q_n]  \\
&= \sum_{m=1}^n [j_m]\otimes_{A_{m}}[\pi_m]+1_{A_{0}} \\
&= 1_{\cK^{S_{n}}\oplus A_{0}}.
\end{aligned}
$$
On the other hand by the induction hypothesis in (\ref{hypothesis}) we have
$$
\begin{aligned}
\Pi_{n}\otimes_{\cK^{S_{n}}\oplus A_{0}} I_n&=[\pi_n]\otimes_{\cK^{m_n}}[j_n] +([q_n]\otimes_{A_{n-1}}[\pi_{n-1}])\otimes_{\cK^{m_{n-1}}} [s_n\circ j_{n-1}] \\ 
&  + ([q_{n-1}\circ q_n]\otimes_{A_{n-2}} [\pi_{n-2}])\otimes_{\cK^{m_{n-2}}} [s_n\circ s_{n-1}\circ j_{n-2}] + \cdots  \\ 
& +([q_2\circ\cdots \circ q_{n-1}\circ q_n]\otimes_{A_{1}} [\pi_{1}])\otimes_{\cK^{m_1}} [s_n\circ s_{n-1}\circ\cdots\circ s_2\circ j_1] \\
& +[q_1\circ q_2\circ\cdots \circ q_n]\otimes_{A_{0}} [s_n\circ s_{n-1}\circ\cdots\circ s_1] \\
&= [\pi_{n}]\otimes_{\cK^{m_n}} [j_n]+[q_n]\otimes_{A_{n-1}}
 +\Big([\pi_{n-1}]\otimes_{\cK^{m_{n-1}}} [j_{n-1}] \\ 
& + ([q_{n-1}]\otimes_{A_{n-2}} [\pi_{n-2}])\otimes_{\cK^{m_{n-2}}} [s_{n-1}\circ j_{n-2}]+ \cdots  \\ 
& +([q_2\circ\cdots \circ q_{n-1}]\otimes_{A_{1}} [\pi_{1}])\otimes_{\cK^{m_{1}}} [s_{n-1}\circ\cdots\circ s_2\circ j_1]
\Big)
\otimes_{A_{n-1}}[s_n] \\
&=  [\pi_{n}]\otimes_{\cK^{m_n}} [j_n]+[q_n]\otimes_{A_{n-1}} \Big(\Pi_{n-1}\otimes_{\cK^{S_{n-1}}\oplus A_{0}} I_{n-1}\Big)\otimes_{A_{n-1}}[s_n] \\
&=  [\pi_{n}]\otimes_{\cK^{m_n}} [j_n]+[q_n]\otimes_{A_{n-1}}[s_n] \\
&= 1_{A_{n}}. 
\end{aligned}
$$
Then $A$ and $\cK^{S_n} \oplus A_{0}$ are $KK$-equivalent by the $KK$-equivalence $\Pi_n$ with inverse $I_n$. 
\end{proof}
\begin{rem}
Note that for $A_0=\mathbb{C}$ we are in the setting of the quantum weighted projective spaces studied by Brzezi\'nski and Szyma\'nski (see their defining extensions in \cite[Proposition~3.2]{BS16}), of which our spaces are obviously a special case. 
\end{rem}

\subsection{The $KK$-equivalence} 
We can now apply Theorem~\ref{thm:KK} to our setting to obtain the desired result.
For $n\geq 2$ we define 
$$
\begin{aligned}
\Pi_n:=&[\pi_n]\oplus([q_n]\otimes_{C(\CP_q^{n-1})}[\pi_{n-1}])\oplus ([q_{n-1}\circ q_n]\otimes_{C(\CP_q^{n-2})} [\pi_{n-2}])\\ 
&\oplus \cdots\cdots
\oplus ([q_2\circ\cdots \circ q_{n-1}\circ q_n]\otimes_{C(\CP_q^1)} [\pi_{1}])\oplus [q_1\circ q_2\circ\cdots \circ q_n]
\end{aligned}
$$
and 
$$
I_n:=[j_n]\oplus [s_n\circ j_{n-1}]\oplus [s_n\circ s_{n-1}\circ j_{n-2}]\oplus\cdots\cdots\oplus [s_n\circ s_{n-1}\circ\cdots\circ s_2\circ j_1]\oplus [s_n\circ s_{n-1}\circ\cdots\circ s_1].
$$
Then $\Pi_n\in KK_h(C(\CP_q^n), \cK^n\oplus\mathbb{C})$ is a $KK$-equivalence with inverse $I_n\in KK_h(\cK^n\oplus\mathbb{C},C(\CP_q^n))$.

\begin{rem}\label{Mortiaequiv}
Let $[ \ell^2(\nn_0)] \in KK(\mathcal{K},\cc)$ denote the class of the natural Morita equivalence. Let $\varphi: \mathbb{C} \to \mathcal{K}(\ell^2(\nn_0))$ be the $*$-homomorphism given by the choice of a rank-one projection, and denote by $[ \varphi ]$ the corresponding class in $KK(\cc, \mathcal{K})$. Note that different choices of rank-one projection yield the same class in $KK$-theory. The two classes are known inverse to each-other. This allows us to write an explicit $KK$-equivalence between ${C(\CP_q^{n})}$ and $\cc^{n+1}$ by
$$ [ \Pi_n ] \otimes_{\cK^n\oplus\mathbb{C}} \big(\bigoplus_{n}[\ell^2(\nn_0)]\oplus [1_\cc]\big)  \in KK(C(\CP_q^n), \mathbb{C}^{n+1}) ,
$$
with inverse $$\big(\bigoplus_{n}[\varphi] \oplus [1_\cc]\big) \otimes_{\cK^n\oplus\mathbb{C}}  [ I_n] \in KK(\mathbb{C}^{n+1},C(\CP_q^n)).$$  
\end{rem}

\section{Splittings and projections}\label{ProjectionsFredholmModules}

In this last section we relate the elements $[j_n], [s_n\circ s_{n-1}\circ\cdots\circ s_{n-k}\circ j_{n-k-1}]\in KK(\cK,C(\CP_q^n))$ from our Theorem~\ref{thm:KK}, to classes of projections in $C(\CP_q^n)$ which generate the group $K_0(C(\CP_q^n))$. To show that the projections are generators of the K-theory, we apply the index pairing with the Fredholm modules defined in \cite[Definition 1]{DAL2010}.

We shall first introduce some notation. Let $\A(\CP^n_q)$ denote the dense $*$-polynomial subalgebra of $C(\CP_q^n)$. For $0\leq i<k\leq n$ let $\epsilon_i^k$ be the array 
$$
\epsilon_i^k:=(\overbrace{0,0,...,0}^{\text{$i$ times}},\overbrace{1,1,...,1}^{\text{$k-i$ times}},\overbrace{0,0,...,0}^{\text{$n-k$ times}}), 
$$
and let $\underline{m}:=(m_1,...m_n)\in\nn^n.$

For $0< k\leq n$ we define a subspace $V_k^n\subseteq\ell^2(\nn^n)$ as the linear span of basis vectors $\ket{m_1,m_2,...,m_n}$ such that $0\leq m_1 \leq m_2 \leq ... \leq m_k$ and $m_{k+1}>m_{k+2}>...>m_n\geq 0$. Here, we let $m_0:=0$. 

For any $0\leq k\leq n$ there exists an irreducible representation $\pi_k^{(n)}: \A(\CP_q^n)\to \B(\ell^2(\nn^n)$ defined on the subspace $V_k^n$ as follows: 
\begin{align*}
\pi_k^{(n)}(z_i)\ket{\underline{m}}&=q^{m_i}\sqrt{1-q^{2(m_{i+1}-m_i+1)}}\ket{\underline{m}+\epsilon_i^k}, &&0\leq i<k, \\
\pi_k^{(n)}(z_k)\ket{\underline{m}}&=q^{m_k}\ket{\underline{m}}, && \\
\pi_k^{(n)}(z_i)&=0, &&i>k\geq 1.
\end{align*}
The representation $\pi_k^{(n)}$ is defined to be zero on the orthogonal complement of $V_k^n$.
For $k=0$ an irreducible representation $\pi_0^{(n)}$ is defined as follows:
\begin{align*}
\pi_0^{(n)}(z_0)\ket{\underline{m}}&=\begin{cases}
\ket{\underline{m}}, & m_1>m_2>...>m_n\geq 0 \\
0, & \text{otherwise}
\end{cases} \\
\pi_0^{(n)}(z_i)\ket{\underline{m}}&=0, i>0. 
\end{align*}
Let $\pi_+^{(n)}(a):=\sum_{\underset{\text{$k$ even}}{0\leq k\leq n}} \pi_k^{(n)}(a)$ and $\pi_-^{(n)}(a):=\sum_{\underset{\text{$k$ odd}}{0\leq k\leq n}} \pi_k^{(n)}(a)$. Then $\pi^{(n)}:=\pi_+^{(n)}\oplus \pi_-^{(n)}$ is a representation of $\A(\CP_q^n)$ on $\mathcal{H}_{(n)}:=\ell^2(\nn^n)\oplus\ell^2(\nn^n)$. 

For each $0\leq t\leq n$  
\[
\mu_t=(\A(\CP^n_q), \, \mathcal{H}_{(t)}, \, \pi^{(t)}, \, \gamma_{(t)}, \, F_{(t)}),
\]
becomes a 1-summable even Fredholm module with $F_{(t)}=\begin{pmatrix} 0 & 1 \\ 1 & 0 \end{pmatrix}$ and $\gamma_{(t)}$ the obvious grading operator. 

The Fredholm modules $\mu_t, t=0,...,n$ are shown to be generators of the $K$-homology group $K^0(\CP^n_q)$ in \cite[Proposition 5]{DAL2010}.

\begin{thm}\label{Projections}
Let $P_0=1$. Then there exists projections $P_l$, $l=1,,,,,n$ in $C(\CP_q^{n})$ such that 
\begin{align*}
\pi_k^{(n)}(P_l)&=0, && k<l \\
\pi_k^{(n)}(P_l)&=\text{The projection onto the subspace } \\ &\hspace{0.3cm}\text{spanned by\ } \{\ket{0,...,0,m_{l+1},...,m_n}\}\cap V_k^{n}, && k\geq l
\end{align*}
where the representation $\pi_k^{(n)}$ and the subspace $V_k^n$ are as above (cp. \cite{DAL2010}). The classes of the projections $P_l$, $l=0,1,...,n$ form a basis $K_0(C(\CP_q^n))$.
\end{thm}
\begin{proof}
We first want to show existence of the projections, they are obtained in a similar way to \cite{BF12} as limits
\[
q^{-2m} \prod_{r=1}^m \frac{q^2\pi_k^{(n)}(z_lz_l^*+z_{l+1}z_{l+1}^*+\cdots +z_nz_n^*)-q^{2(r+1)}}{1-q^{2r}} \xrightarrow{m\to \infty} \pi_k^{n}(P_l),
\]
where by $m\to\infty$ we refer to norm convergence.

It is clear that if $k<l$ then $\pi_k^{(n)}(z_lz_l^*+z_{l+1}z_{l+1}^*+\cdots +z_nz_n^*)=0$. For $k\geq l$ we have 
\[
\begin{aligned}
&\pi_k^{(n)}(z_lz_l^*+z_{l+1}z_{l+1}^*+\cdots +z_nz_n^*)\ket{m_1,...,m_n} \\ &=
(q^{m_l}(1-q^{2(m_{l+1}-m_l)})+\cdots + q^{2(m_k-1)}(1-q^{2(m_k-m_{k-1})})+q^{2m_k})\ket{m_1,...,m_n} \\ & = q^{2m_l} \ket{m_1,...,m_n}
\end{aligned}
\]
when $\ket{m_1,...,m_n}\in V_k^n$ otherwise $0$. Then 
\[
\begin{aligned}
& q^{-2m} \prod_{r=1}^m \frac{q^2\pi_k^{(n)}(z_lz_l^*+z_{l+1}z_{l+1}^*+\cdots +z_nz_n^*)-q^{2(r+1)}}{1-q^{2r}}\ket{m_1,...,m_n} \\ &= q^{-2m} \prod_{r=1}^m \frac{q^{2(m_l+1)}-q^{2(r+1)}}{1-q^{2r}}  \ket{m_1,...,m_n} \\
&\xrightarrow{m\to \infty}
\begin{cases}
\ket{m_1,...,m_n}, & m_l=0 \ \text{and} \ \ket{m_1,...,m_n}\in V_k^n \\
0, & \text{otherwise}
\end{cases}.
\end{aligned}
\]
Since $\ket{m_1,...,m_n}\in V_k^n$ we have $0\leq m_1 \leq m_2 \leq \cdots \leq m_k $ but since $k\geq l$ we have $m_i=0$ for $i\leq l$ and we get the projections $P_l$, $l=1,2,...,n$. 

This allows us to calculate the index pairing between the $K$-theory and the $K$-homology. Each projection gives a class in the $K$-theory given by $[P_l]=[(C(\CP_q^n),\psi_l,0)]\in KK(\mathbb{C}, C(\CP_q^n))$ where $\psi_l(1)=P_l$. We now wish to pair these with the classes $[\pi_t]\in KK(C(\CP_q^n), \mathbb{C})$, $t=0,1,...,n$ in \cite{DAL2010}. 
The product is given by 
\[
[P_l]\otimes_{C(\CP_q^n)}[\pi_t]=[(l^2(\mathbb{N}_0^t)_+\oplus l^2(\mathbb{N}_0^t)_-, \pi^{(t)}\circ \psi_l, F, \gamma) ].
\]
If $t<l$ then $\pi^{(t)}\circ \psi_l(1)=0$ since $\pi_i^{(t)}(P_l)=0, i\leq t$ hence the product is 0. In the case where $l<t$ we obtain that $[\pi^{(t)}\circ \psi_l(1), F]=0$ by the following result: 
\begin{align*}
\sum_{\underset{\text{k even}}{0\leq k\leq t}} \pi_k^{(t)}(P_l)-\sum_{\underset{\text{k odd}}{0\leq k\leq t}} \pi_k^{(t)}(P_l) = \sum_{\underset{\text{k even}}{l\leq k\leq t}} \pi_k^{(t)}(P_l)-\sum_{\underset{ \text{k odd}}{l\leq k\leq t}} \pi_k^{(t)}(P_l)=0
\end{align*}
The above follows since for every $\pi_k^{(t)}$ there are two other representation, namely $\pi_{k-1}^{(t)}$ and $\pi_{k+1}^{(t)}$, defined on orthogonal subspaces which are also non-zero on some part of $V_k^n$. We then obtain a degenerate module, hence the product is 0. 

For $l=t$ and $l$ even we have 
\[
\begin{aligned}
\ [P_l]\otimes_{C(\CP_q^n)}[\pi_l]&=[(l^2(\mathbb{N}_0^l)_+\oplus l^2(\mathbb{N}_0^l)_-, \pi^{(l)}\circ \psi_l, F, \gamma) ]
\\ 
&=[(\pi_l^{(l)}(P_l)l_2(\mathbb{N}_0^l)\oplus 0, M_{\mathbb{C}}, F, \gamma)] \\
&=[(\mathbb{C}, \id_{\mathbb{C}}, 0)] = [1_{\mathbb{C}}]
\end{aligned}
\]
We get a similar result if $l$ is odd. Hence it follows that $[P_l]\otimes_{C(\CP_q^n)}[\pi_t]=[1_{\mathbb{C}}]$ if $l=t$ otherwise it is $0$. Since the matrix with entries $a_{lt}=[P_l]\otimes_{C(\CP_q^n)}[\pi_t]$ is the identity matrix and is then invertible we get that $P_l, l=0,1,...n$ generate the $K$-theory. 
\end{proof}

\begin{rem}
We remark that the projections constructed here are very similar in fashion to the faithful irreducible representations of quantum teardrop presented in \cite[Section~2]{BF12}. Those were later used in \cite[Section~7.4]{AKL16} to prove an explicit $KK$-equivalence result.
\end{rem}
\begin{prop}
Let 
$
\pi: C(S_q^{2n+1})\to B(l^2(\nn^n\times \zz))
$ be the faithful representation of the Vaksman--Soibelman sphere defined in \eqref{eq:repVS}, and $\rho: C^*(L_{2n+1})\to B(l^2(\nn^n\times \zz))$ the faithful graph algebra representation in \eqref{RepRho}. 
The projections $P_{n-k-1}$ satisfy  
\begin{equation}
\label{eq:relProj}    
\begin{aligned}
\pi(P_{n-k-1})=\rho(P_{w_{n-k}}+P_{w_{n-k+1}}+P_{w_{n-k+2}}+ \cdots + P_{w_{n+1}}),
\end{aligned}
\end{equation}
for $k=0,1,...,n-1$. 
Moreover, up to Morita equivalence, they define the same classes in $KK$-theory as $[s_n\circ s_{n-1}\circ\cdots\circ s_{n-k}\circ j_{n-k-1}].$ More precisely, 
if $\varphi : \mathbb{C} \to \mathcal{K}$ is the $*$-homomorphism from Remark \ref{Mortiaequiv}, we have the following equalities of classes in $KK(\mathbb{C}, \CP^n_q)$:
\begin{equation}
\label{KK1}
\begin{aligned}[]
& [P_n]=[\varphi]\otimes_{\mathcal{K}} [j_n],\qquad  \mbox{and}\\   & [P_{n-k-1}]=[\varphi]\otimes_{\mathcal{K}} [s_n\circ s_{n-1}\circ\cdots\circ s_{n-k}\circ j_{n-k-1}], \qquad k=0, \dots, n-1.
\end{aligned}\
\end{equation}
\end{prop}

\begin{proof} We start our proof by recalling our labeling convention on the graph $F_n$ underlying the algebra $C(\CP^n_q)$: vertices will be denoted by $w$, and edges by $f$. We will do so independently of the dimension $n$ of the space.

Let $0\leq k<n$, we denote by $I_{\{w_{n-k}\}}$ the 2-sided closed ideal generate by the projection $P_{w_{n-k}}$. Then 
$$
j_{n-k-1}:I_{\{w_{n-k}\}}\to C^*(F_{n-k-1})
$$
and 
$$
\begin{aligned}
s_n&\circ s_{n-1}\circ\cdots\circ s_{n-k}\circ j_{n-k-1}(P_{w_{n-k}})\\ &=
s_n\circ s_{n-1}\circ\cdots\circ s_{n-k-1}(P_{w_{n-k}}+P_{w_{n-k+1}}) \\
&= s_n\circ s_{n-1}\circ\cdots\circ s_{n-k-2}(P_{w_{n-k}}+P_{w_{n-k+1}}+P_{w_{n-k+2}}) \\
&= P_{w_{n-k}}+P_{w_{n-k+1}}+P_{w_{n-k+2}}+ \cdots + P_{w_{n+1}}\in C(\CP_q^n),
\end{aligned}
$$
 Any element in $I_{\{w_{n-k}\}}$ takes the form $S_{\alpha}S_{\beta}^*$ where $\alpha$ and $\beta$ are finite paths in $F_{n-k-1}$ such that $r(\alpha)=r(\beta)=w_{n-k}$ by \cite[Lemma 3.2]{BHRS02}. 

By the structure of the graph and the splitting in \eqref{splitting},  we have 
$$
\begin{aligned}
s_{n-k}(S_{\alpha}S_{\beta}^*)
&=s_{n-k}(S_{\alpha'}P_{w_{n-k}}S_{\beta'}^*)= S_{\alpha'}(w_{v_{n-k}}+P_{w_{n-k+1}})S_{\beta'}^*
\end{aligned}
$$
where $\alpha=\alpha^{\prime} f_{i_1,n-k}^{m_1}, \beta=\beta^{\prime} f_{i_2,n-k}^{m_2}$. Then the image of $I_{\{w_{n-k}\}}$ under $s_{n-k}$ consists of all $S_{\mu}S_{\nu}^{*}$ such that $r(\mu)=r(\nu)\in \lbrace w_{n-k},w_{n-k+1}\rbrace$ which is precisely the ideal generated by the sum $P_{w_{n-k}}+P_{w_{n-k+1}}$ in $C(\CP_q^{n-k})$ by \cite[Lemma 3.2]{BHRS02}. Continuing like this we obtain that the image of $I_{\{w_{n-k}\}}$ under $s_n\circ s_{n-1}\circ\cdots\circ s_{n-k}\circ j_{n-k-1}$ is the ideal generated by $P_{w_{n-k}}+P_{w_{n-k+1}}+P_{w_{n-k+2}}+ \cdots + P_{w_{n+1}}$ in $C(\CP_q^n)$. 

Consider the projections $P_l,l=0,...,n$ in the representation $\pi$ of $C(S_q^{2n+1})$, restricted to $C(\CP^{n}_q)$. Then 
\begin{equation}\label{reppi}
\begin{aligned}
q^{-2n}& \prod_{r=1}^n \frac{q^2\pi(z_lz_l^*+z_{k+1}z_{k+1}^*+\cdots +z_nz_n^*)-q^{2(r+1)}}{1-q^{2r}}\xi(k_1,...,k_n,m) \\&=\frac{q^2q^{2(k_1+\cdots +k_l)}-q^{2(r+1)}}{1-q^{2r}}\xi(k_1,...,k_n,m) \\
&=\begin{cases}
\xi(k_1,...,k_{n},m) & \text{if} \ k_1+\cdots + k_l=0 \\
0 & \text{otherwise}
\end{cases}
\end{aligned}
\end{equation}

Then $\pi(P_l)$ is the projection onto the subspace spanned by 
$$\{\xi(0,\cdots,0,k_{l+1},\cdots,k_n,m)|\ k_i\in \mathbb{N}, m\in \zz \}.$$ 
Moreover, under the representation $\rho$ of the graph $C^*$-algebra $C^*(L_{2n+1})$, which descends to $C^*(F_{n})$, we have
$$
\begin{aligned}
\rho&(1-(P_{w_1}+P_{w_2}+ \cdots + P_{w_l}))\xi(k_1,...,k_n,m) \\ &=
(1-((1-\delta_{k_1,0})+\delta_{k_1,0}(1-\delta_{k_2,0})+\delta_{k_1,0}\delta_{k_2,0}(1-\delta_{k_3,0}) \\
&\hspace{0.2cm}+\cdots + \delta_{k_1,0}\delta_{k_2,0}\cdots\delta_{k_{k-1},0}(1-\delta_{k_l,0})))\xi(k_1,...,k_n,m) \\
&= \delta_{k_1,0}\delta_{k_2,0}\delta_{k_l,0}\xi(k_1,...,k_n,m).
\end{aligned}
$$
Hence $\rho(1-(P_{w_1}+P_{w_2}+ \cdots + P_{w_l}))=\pi(P_l)$ and
$$
\rho(P_{w_{n-k}}+P_{w_{n-k+1}}+ \cdots + P_{w_{n+1}}) =\rho(1-(P_{w_1}+\cdots + P_{w_{n-k-1}}))
= \pi(P_{n-k-1}).
$$

Let $\psi$ be the irreducible representation in \eqref{eq:repVS}. In order to prove \eqref{KK1} for $P_n$ we note that $\psi(P_n)$ is the projection onto the subspace spanned by the vector $\zeta(0,...,0)$, which follows by a calculation similar to the one in \eqref{reppi}. Since 
$$
\psi\circ \pi^{-1}\circ\rho(P_{w_{n+1}})=\psi(P_n)
$$
it follows that $P_{w_{n+1}}\in C^*(F_n)$ is a rank-one projection by considering the representation $\psi\circ \pi^{-1}\circ\rho$. We can set $\varphi(1)=P_{w_{n+1}}$ since, as mentioned earlier, the choice of rank-one projection does not affect the class in $KK$-theory. Moreover, 
by \eqref{eq:relProj}, we have 
$$\rho(j_n\circ \varphi(1))=\rho(P_{w_{n+1}})=\pi(P_n).$$
Since $[\varphi]\otimes_\cK [j_n]=[(C(\CP_q^n),j_n\circ\varphi,0)]\in KK(\cc,C(\CP_q^n))$ we obtain \eqref{KK1}. 

Similarly we have that $P_{w_{n-k}}$ is a rank one projection in $C^*(F_{n-k-1})$. Hence for each $k=0,1,...,n-1$ we can set $\varphi(1)=P_{w_{n-k}}$. Then 
$$
\begin{aligned}
\rho&(s_n\circ s_{n-1}\circ\cdots\circ s_{n-k}\circ j_{n-k-1}\circ\varphi(1)) \\&=\rho(P_{w_{n-k}}+P_{w_{n-k+1}}+P_{w_{n-k+2}}+ \cdots + P_{w_{n+1}})= \pi(P_{n-k-1})
\end{aligned}
$$
by \eqref{eq:relProj}.
Hence 
\[
[\varphi]\otimes_{\mathcal{K}} [s_n\circ s_{n-1}\circ\cdots\circ s_{n-k}\circ j_{n-k-1}]=[s_n\circ s_{n-1}\circ\cdots\circ s_{n-k}\circ j_{n-k-1}\circ \varphi]=[P_{n-k-1}],
\]
as desired.
\end{proof}

Note that it follows directly from the identification in \eqref{eq:relProj}, that $P_l, l=0,1,...,n$ generates $K_0(C(\CP_q^{n}))$ which was shown in Theorem \ref{Projections} using the index pairing with Fredholm modules.


\begin{thebibliography}{99.}%
%
%

\bibitem{A18} F. Arici, \emph{ Gysin exact sequences for quantum weighted lens spaces},
MATRIX Book Ser.{\bf 1}, Springer, Cham, (2018), 251–262. 


\bibitem{AKL16}
  F.~Arici, J.~Kaad, G.~Landi, \emph{Pimsner algebras and Gysin sequences 
  from principal circle actions}, J. Noncommut. Geom. \textbf{10} (2016), 29--64.
  
\bibitem{BHRS02} T. Bates, J.H. Hong, I. Raeburn and W. Szymański, 
{\em The ideal structure of the $C^*$-algebras of infinite graphs},
Illinois J. Math. \textbf{46} (2002), no 4, 1159-1176. 

\bibitem{BPRS00} T. Bates, D. Pask, I. Raeburn and W. Szymański, {\em The \Cs -algebras of row-finite graphs}, New
York J. Math. \textbf{6} (2000), 307-324.


\bibitem{Bl98}
B.~Blackadar, \textit{$K$-Theory for Operator Algebras}, 2nd edition,
Cambridge University Press, 1998. 
  
\bibitem{BF12}
{T.~Brzezi{\'n}ski, S.A.~Fairfax}, \emph{Quantum
  teardrops}, Comm. Math. Phys. {\bf 316} (2012), 151--170.  



\bibitem{BS16} T. Brzeziński and W. Szymański, {\em The $C^*$-algebras of quantum lens and weighted projective spaces}, J. Noncommut. Geom. \textbf{12} (2018), 195-215.

\bibitem{Cu87}
J. Cuntz, \emph{A new look at $KK$-theory}, $K$-Theory {\bf 1} (1987), no. 1, 31--51. 

\bibitem{DAL2010} F. D'Andrea \& G. Landi, 
{\em Bounded and unbounded Fredholm modules for quantum projective spaces,} J. K-Theory {\bf 6} (2010), 231--240.

\bibitem{DT02} D. Drinen and M. Tomforde,
{\em Computing K-theory and Ext for graph $C^*$-algebras}, Illinois J. Math. \textbf{46} (2002), no. 1, 81–91.  
 
  

\bibitem{FLR00} N. J. Fowler, M. Laca and I. Raeburn,   
{\em The $C^*$-algebras of infinite graphs}, 
Proc. Amer. Math. Soc. {\bf 128}  (2000),  2319--2327. 

\bibitem{Hig87}
N. Higson, \emph{A characterization of $KK$-theory}, Pacific J. Math. {\bf 126} (1987), no. 2, 253--276.


\bibitem{HS02} J. H. Hong and W. Szyma\'{n}ski, 
\emph{Quantum spheres and projective spaces as graph algebras}, 
Commun. Math. Phys. {\bf 232} (2002), 157--188.

\bibitem{HS03}
J.H.~Hong, W.~Szyma{\'n}ski, \textit{Quantum lens spaces and graph algebras}, Pac. J. Math. {\bf 211} (2003), 249--263.

\bibitem{JT91} K.K. ~Jensen, K. ~Thomsen, \emph{Elements of $KK$-theory}. Springer, (1991). 

\bibitem{Ka78} 
{M.~Karoubi}, \emph{K-Theory: an Introduction}. Grund. der math. Wiss.  226, Springer, (1978).

\bibitem{Kas80}
G. G. Kasparov, {\em The operator $K$-functor and extensions of $C\sp{\ast}$-algebras}, Izv. Akad. Nauk SSSR Ser. Mat. {\bf 44} (1980), no. 3, 571--636, 719. 

\bibitem{La05} G. Landi, {\em Noncommutative spheres and instantons}, Quantum field theory and noncommutative geometry, 3–56, Lecture Notes in Phys., \textbf{662}, Springer, Berlin, (2005).

\bibitem{NesTus12}
S. Neshveyev\ and\ L. Tuset, Quantized algebras of functions on homogeneous spaces with Poisson stabilizers, Comm. Math. Phys. {\bf 312} (2012), no.~1, 223--250. 




\bibitem{PR96} D. Pask and I. Raeburn, {\em On the K-theory of Cuntz-Krieger algebras}, Publ. RIMS, Kyoto Univ. \textbf{32} (1996). 

\bibitem{R05} I. Raeburn, {\em Graph algebras}, CBMS Regional Conference Series in Mathematics,
103, Published for the Conference Board of the Mathematical Sciences,
Washington, DC; by the American Mathematical Society, Providence, RI,
(2005), ISBN: 0-8218-3660-9. 

\bibitem{RS04} I. Raeburn and W. Szymański, {\em Cuntz-Krieger algebras of infinite graphs and matrices}, Trans. Amer. Math. Soc. \textbf{356} (2004), 39-59.

 \bibitem{RoSch87}
J. Rosenberg\ and\ C. Schochet, {\em The K\"{u}nneth theorem and the universal coefficient theorem for Kasparov's generalized $K$-functor}, Duke Math. J. {\bf 55} (1987), no. 2, 431--474. 


\bibitem{S02} W. Szymański, {\em On semiprojectivity of $C^*$-algebras of directed graphs}, 
Proc. Amer. Math. Soc. \textbf{130} (2002), 1391-1399.

\bibitem{VS91}
L.~Vaksman, Ya.~Soibelman, \emph{The algebra of functions on the quantum
 group $\textup{SU}(n+1)$ and odd-dimensional quantum spheres}, Leningrad Math. J. {\bf 2} (1991), 
 1023--1042.


\end{thebibliography}
\end{document}